\documentclass[a4paper,reqno]{amsart}
\pdfoutput=1
\usepackage[T1]{fontenc}
\usepackage{lmodern}
\usepackage{dsfont}

\usepackage{amsmath,amsfonts,amssymb,amsthm,stmaryrd,bbm,graphicx,mathtools,enumerate}
\usepackage{mathrsfs}
\usepackage[utf8]{inputenc}  

\usepackage[english]{babel}
\mathtoolsset{showonlyrefs} 
\usepackage[tracking,spacing,kerning,babel]{microtype}
\usepackage{wasysym} 
\usepackage{color}
\usepackage{xcolor}
\usepackage{framed} 

\usepackage{wrapfig} 

\usepackage[normalem]{ulem}
\usepackage{soul}

\usepackage[final]{hyperref}   
\hypersetup{
    linktoc=page,
    linkcolor=red,          
    citecolor=blue,        
    filecolor=blue,      
    urlcolor=cyan,
   colorlinks=true           
}

\usepackage{lipsum} 

\usepackage[toc,page]{appendix}

\DeclareRobustCommand{\SkipTocEntry}[5]{}

\usepackage{multicol}

\newtheoremstyle{mine}
{\baselineskip}
{\baselineskip}
{\itshape}
{
}
{\bfseries}
{.}
{.5em}
{#1 #2\ifx#3\relax\else~(#3)\fi}

\theoremstyle{mine}

\newtheorem{theorem}{Theorem}
\numberwithin{theorem}{section}
\newtheorem{corollary}[theorem]{Corollary}
\newtheorem{proposition}[theorem]{Proposition}

\newtheorem{lemma}[theorem]{Lemma}

\newtheorem{definition}[theorem]{Definition}

\newtheorem{OP}{Open Problem} 
\numberwithin{equation}{section}

\theoremstyle{remark}
\newtheorem{remark}{Remark}

\theoremstyle{remark}

\colorlet{shadecolor}{blue!10}

\newcommand{\margin}[1]{\textcolor{magenta}{*}\marginpar[\textcolor{magenta} {  \raggedleft  \footnotesize  #1 }  ]{ \textcolor{magenta} { \raggedright  \footnotesize  #1 }  }}
\newcommand{\note}[2]{ \hskip 2cm  \textcolor{blue}{\large \bf #1 :}   \vline\,\vline \hskip 0.5 cm \parbox{10 cm}{ #2}  }
\newcommand{\margingreen}[1]{\textcolor{darkgreen}{*}\marginpar[\textcolor{darkgreen} {  \raggedleft  \footnotesize  #1 }  ]{ \textcolor{magenta} { \raggedright  \footnotesize  #1 }  }}
\renewcommand{\margin}[1]{}
\renewcommand{\margingreen}[1]{}
\renewcommand{\note}[2]{}

\let\qed=\QED

\newcommand{\mbb}{\mathbb} 
\newcommand{\mrm}{\mathrm}

\newcommand{\R}{\mathbb{R}}

\newcommand{\Z}{\mathbb{Z}}

\renewcommand{\S}{\mathbb{S}}
\newcommand{\B}{\mathbb{B}}

\def \eps{\varepsilon}

\def \curl{\operatorname{curl}}
\def \Flux{\operatorname{Flux}}
\def\div{\operatorname{div}}

\def\dist{\mathrm{dist}}

\def\floor#1{\lfloor{#1}\rfloor}

\def\P{\mathbb{P}} 
\def\E{\mathbb{E}} 
\def\md{\mid}

\def\FK#1#2#3{{\def\md{\bigm| } \P_{#1}^{\,#2}  \bigl[  #3 \bigr]}}

\def \p {{\partial}}

\def\<#1{\langle #1\rangle}

\newcommand{\green}[1]{{\color{green}#1}}

\definecolor{darkgreen}{rgb}{0,0.6,0.05}
\newcommand{\dgreen}[1]{{\color{darkgreen}#1}}

\def\bi{\begin{itemize}}  
\def\ei{\end{itemize}}
\def\bnum{\begin{enumerate}} 
\def\enum{\end{enumerate}}
\def\ni{\noindent}
\def\bf{\bfseries}

\def\DE{\mathrm{DE}}

\title[Lattice based harmonic maps]{Regularity theory for lattice based harmonic maps and low temperature Gibbs measures of \\
the spin-$O(N)$ model}
\author{Christophe Garban}
\author{Nathan de Montgolfier}
\date{}

\address
{Universit\'e Claude Bernard Lyon 1, CNRS UMR 5208, Institut Camille Jordan, 69622 Villeurbanne, France}
\email{christophe.garban@gmail.com}

\address
{Universit\'e Claude Bernard Lyon 1, CNRS UMR 5208, Institut Camille Jordan, 69622 Villeurbanne, France}
\email{demontgolfier@math.univ-lyon1.fr}

\newcommand{\newsection}[2][]{%
  \if\relax\detokenize{#1}\relax
    \section{#2}%
  \else
    \section{\for{toc}{#1}\except{toc}{#2}}%
  \fi
}

\begin{document}

\begin{abstract}
The purpose of this note is twofold. First, we revisit the classical regularity theory of harmonic maps in the setting of maps defined on a lattice $\Z^d$ instead of $\R^d$ or some other manifold. The target space of our harmonic maps will be spheres $\S^{N-1}, N\geq 2$. A harmonic map $\sigma : \Z^d \to \S^{N-1}$ is a map which in any finite subgraph $\Lambda$ minimizes the internal Dirichlet energy $\frac 1 2 \sum_{i \sim j \in \Lambda} \|\sigma_i -\sigma_j\|^2$ given the boundary condition $\sigma_{\md \p  \Lambda}$. 

One key difference between our setting and the continuum is that our main Hölder regularity statement in a large domain $\Lambda\subset \Z^{d=2}$ does not require any regularity assumption on the map $\sigma$ restricted to the boundary of the domain. For example, when $d=2$ and $N=3$, the Hölder regularity induced in the bulk of the domain holds uniformly in all boundary configurations $\sigma \in (\S^{2})^{\p \Lambda}$ (this is in contrast with the continuum which requires extensions in suitable Sobolev spaces). 

The fact we obtain results uniformly in the boundary conditions allows us to extend our results for positive temperatures. Namely, as $T \searrow 0$, we obtain a local regularity theory for all infinite volume Spin $O(N)$ Gibbs measures, at least when $d=2$ and $N\geq 3$. 

We also discuss implications of our analysis in dimensions $d\geq 3$. 
\end{abstract}

\maketitle
\tableofcontents

\section{Introduction}

Harmonic maps from $\R^d$ to compact target manifolds such as $\S^{N-1}, N \geq 2$ have been extensively studied in the literature since the seminal work of Morrey \cite{morrey1948problem}.
One key aspect of these harmonic maps is their regularity (or lack of regularity) as studied in the works \cite{morrey1948problem,schoen1982regularity,hardt1988stable,helein1991regularite,evans1991partial,riviere2007conservation}. When the base space is continuous, say $\Omega \subset \R^d$, one may distinguish at least three definitions of {\em harmonic maps:} 
\bnum
\item {\em Minimizing harmonic maps:} maps $u \in W^{1,2}(\Omega, N)$ that minimize the Dirichlet energy $E(u) = \frac{1}{2} \int_\Omega |\nabla u|^2 \, dx$ among all Sobolev maps sharing the same boundary data.
\item {\em Stationary harmonic maps:} maps that are critical points of the Dirichlet energy under both target variations and domain variations (diffeomorphisms of $\Omega$), yielding a divergence-free stress-energy tensor and monotonicity formulas.
\item {\em Weakly harmonic maps:} maps in $W^{1,2}(\Omega, N)$ that are critical points of the energy with respect to target variations, i.e. satisfying the Euler--Lagrange system $-\Delta u = A(u)(\nabla u, \nabla u)$ in the sense of distributions.
\enum
One has then $(1) \subset (2) \subset (3)$ and the regularity property of harmonic maps typically depend on which class is considered. We refer to the monographs \cite{helein2002harmonic,lin2008analysis} for background on harmonic maps and their regularity properties.

\medskip

The first purpose of this paper is to revisit the regularity properties of harmonic maps when the base space is a lattice $\Z^d$ instead of $\R^d$ (or a finite subgraph $\Lambda \subset \Z^d$ subject to boundary conditions rather than a domain $\Omega \subset \R^d$). Here is a precise definition. First, if $\Lambda\subset \Z^d$, we shall denote by $\p \Lambda$  its inner boundary 
\begin{align*}  
\p \Lambda & := \{ i \in \Lambda \text{ s.t. } \exists j\in \Lambda^c \text{ with } i\sim j \}\,.
\end{align*}

\begin{definition}[Lattice based harmonic maps]  
Let $d\geq 2$, $N\geq 2$ and $\Lambda \subset \Z^d$ be a finite subset. Given $\bar \sigma : \p \Lambda \to \S^{N-1}$, a harmonic map in $\Lambda$ with boundary conditions $\bar \sigma$ is any configuration $\sigma \in (\S^{N-1})^\Lambda$ which belongs to 
\begin{align*}  
\mathrm{argmin} \left\{ \sum_{i,j \in \Lambda, \, i\sim j } \|\sigma_i - \sigma_j\|_2^2  \;\;  \Big|  \;\;   \sigma \in (\S^{N-1})^\Lambda, \text{ s.t. } \sigma_{\p \Lambda} = \bar \sigma \right\}\,.
\end{align*}
(N.B. as opposed to linear harmonic functions (with values in $\R^N$), uniqueness is not always satisfied). 

In the case of the full lattice $\Z^d$, we will say that a map $\sigma : \Z^d \to \S^{N-1}$ is a (global) harmonic map if for any finite $\Lambda \subset \Z^d$, $\sigma$ is a harmonic map in $\Lambda$ with boundary data $\sigma_{| \p \Lambda}$. 
\end{definition}

\begin{figure}[!htp]
\begin{center}
\includegraphics[width=1\textwidth]{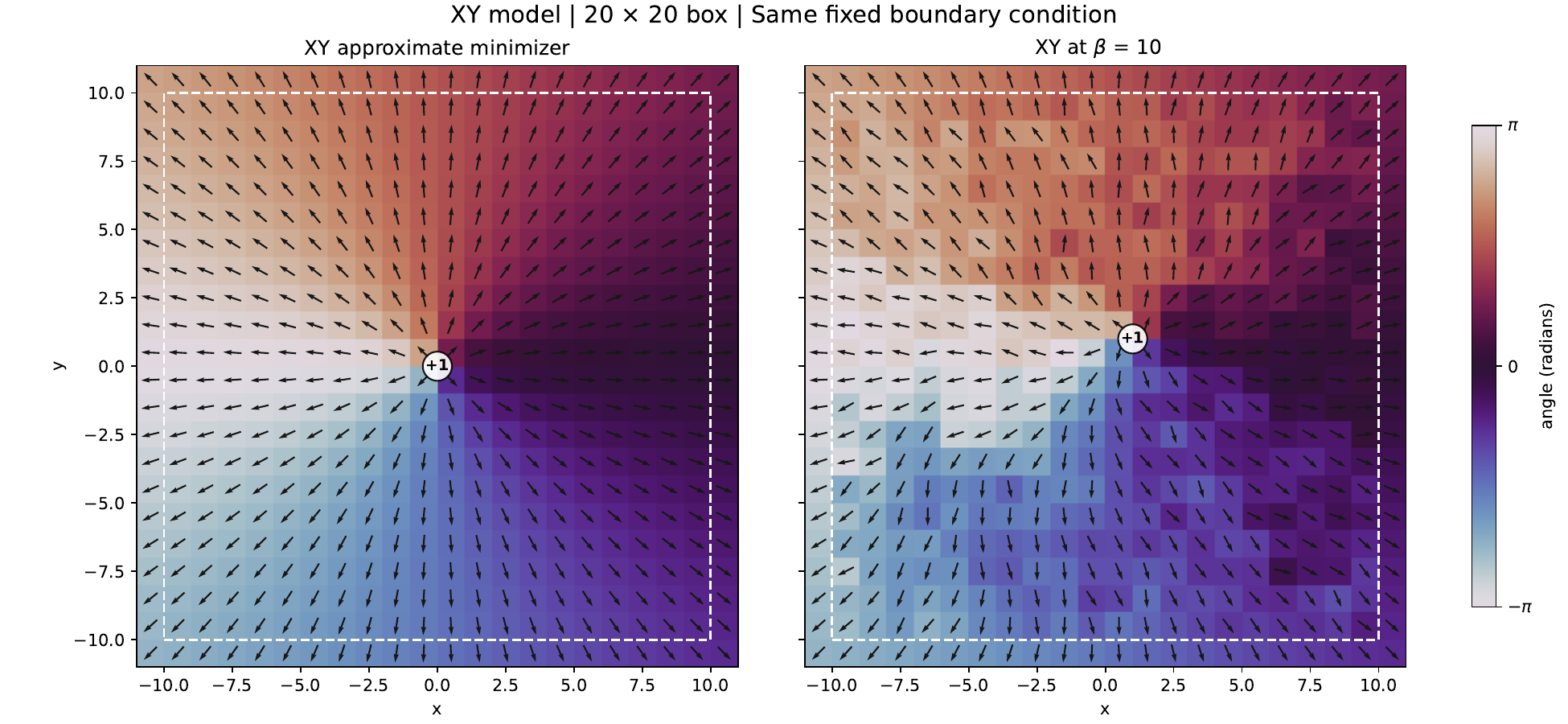}
\end{center}
\caption{On the left, an approximation of a ground state of the XY model, on the right a simulation of the XY model at low temperature ($\beta = 10$). The code for these simulations was obtained by using ChatGPT-6 Astra. }\label{f.XY}
\end{figure}

One major difference in our present lattice setting is that our results will not require any condition on the boundary conditions (no need of Sobolev traces at $\p \Omega$ etc.) and will hold uniformly over all boundary conditions $\sigma \in (\S^{N-1})^{\p \Lambda}$. To handle those, in the case $d=2$ which is the main motivation of this work, our proof will be strongly inspired by a combination of \cite{morrey1948problem} and \cite{hardt1988stable}.

\smallskip

If $d\geq 2$ and $n\geq 1$, $\Lambda_n$ will correspond to the finite box $\{-n,\ldots,n\}^d \subset \Z^d$. 

Our first main result holds in the planar $d=2$ case:
\begin{theorem}\label{th.H}
Let $d=2$ and $N\geq 3$. Then there exists $\alpha=\alpha(N)>0$ such that uniformly in $n\geq 1$ and $\bar \sigma \in (\S^{N-1})^{\p \Lambda_n}$, any harmonic map extension $\sigma$ of $\bar \sigma$ inside $\Lambda_n$ satisfies the following Hölder regularity bound. (N.B. There may exist several harmonic map extensions of $\bar \sigma$ and the bound is valid for all for them). 
\begin{align}\label{e.H}
\| \sigma(x) - \sigma(y)\|_{2, \R^N} \leq 2 \left(\frac{ \dist(x,y) }{ \dist(x, \p \Lambda_n) \wedge \dist(y, \p \Lambda_n)  } \,
\wedge 1 \right)^\alpha
\end{align}
\end{theorem}

\ni
Theorem \ref{th.H} has the following immediate consequence. 
\begin{corollary}
When $d=2$ and $N\geq 3$, the only global harmonic maps are the constant ones. 
\end{corollary}

\ni
Interestingly, this result is wrong when $N=2$. Namely,
\begin{proposition}\label{pr.ExS1}
When $d=2$ and $N=2$, there exist non-constant global harmonic maps $\sigma : \Z^2 \to \S^1$.
Moreover, one can choose $\sigma$ such that for any $n\ge 1$, $\DE^{\Lambda_n}(\sigma) \ge c \log(n)$ from numerical constant $c> 0$.
\end{proposition}
\ni
The proof of this proposition, which we found with the help of chatGPT, is given in Section \ref{s.VortexS1}.

\medskip


We now turn to case of sphere-valued harmonic maps when the dimension of the base lattice $\Z^d$ is any $d \geq 2$ and in particular on the control of their Dirichlet energy. Let us first set up some notations. 
Let $\Lambda \subset \Z^d$. If $\psi : \Lambda \to \R^{N}$ (it may not be $\S^{N-1}$ valued) and $A \subset \Lambda $, we denote the Dirichlet energy of $\psi$ measured inside the set $A$ by
\begin{align}\label{e.DA}
\DE_A = \DE_A(\psi) := \frac 1 2 \sum_{e \in E(A)} \| \nabla_e \psi \|^2\,,
\end{align}
where $E(A) = \{ \{i,j\} \mid i \in A, j \in A, i \sim j  \}$.
Our second main result is.

\begin{theorem}\label{th.DE}
Let $d\geq 2$ and $N\geq 2$. There exists a universal constant $C = C(d,N)$ such that for any $n\geq 1$,  uniformly over boundary condition $\bar \sigma \in  (\S^{N-1})^{\p \Lambda_n}$,  harmonic maps $\sigma$ with boundary data $\bar \sigma$ satisfy

 \begin{align}  
\text{DE}_{\Lambda_{n/2}}(\sigma) \le \begin{cases} C n^{d-2} & \text{if } N \ge 3 \\  C n^{d-2} \log^2(n) & \text{if } N = 2 \,.\end{cases} 
\end{align}

Furthermore, when $d=2$ and $N \ge 3$, we also have the following estimate at smaller scales $L  \leq \frac n 2$:

\begin{align}\label{e.DElog}
\DE_{\Lambda_L }(\sigma) \le \frac{C}{\log(n/L)}\,.
\end{align}
 
\end{theorem}

\begin{remark}  
As we shall see in Section \ref{sec:holder_reg}, the estimate~\eqref{e.DElog}, which is uniform over boundary conditions, will be one of the steps of the proof of Theorem \ref{th.H}. 
\end{remark}

\begin{remark}  
As we will see in Section \ref{s.VortexS1}, in the case $N=2$, one can construct harmonic maps in $\Lambda_n \subset \Z^d$ whose Dirichlet energy inside $\Lambda_{n/2}$ is greater than $n^{d-2} \log(n)$. It is not clear to us whether the factor $\log^2(n)$ is an artefact of our proof or whether it is in fact optimal.  See Open Problem \ref{op.logn}. 
\end{remark}

\begin{remark}  
When $d\geq 3$, a direct consequence of the above result is that harmonic maps do not produce two-dimensional ``cliffs'' with large gradient in the bulk. 
\end{remark}

As mentioned earlier, the above two results (Theorem \ref{th.H} and Theorem \ref{th.DE}) have well-known analogues in the literature on the regularity of harmonic maps with a continuous base space $\sigma: \R^d \to \S^{N-1}$. (See the references given above). 
\medskip

We now move to our third main result which is more probabilistic in nature and does not have -- to our knowledge -- existing analogues in the case of a continuous base space. It involves a small temperature $T$ (and the small thermal fluctuations it induces). Our main objective is to obtain local regularity results for the low temperature Gibbs states of the so-called spin-$O(N)$ model which we shall now define (the cases $N=2$ and $N=3$ correspond respectively to the celebrated $XY$ and classical Heisenberg model). 

\medskip


The spin-$O(N)$ model is an important statistical physics model which assigns random spins $\sigma_i \in \S^{N-1}$ to the vertices $i \in \Z^d$. Its name comes from the underlying symmetry group $O(N)$ acting on $\S^{N-1}$ (and justifies our choice of notation $\S^{N-1}$ instead of $\S^N$ through this work).  It can easily be defined on a finite box $\Lambda \subset \Z^d$ subject to boundary conditions $\bar \sigma \in (\S^{N-1})^{\p \Lambda}$. 
As usual in statistical physics, the parameter $\beta$ corresponds to the \textbf{inverse temperature} ($\beta = T^{-1}$) and large values of $\beta$ correspond to low temperature states. The link with harmonic maps is clear when $\beta \nearrow \infty$ (or equivalently $T \searrow 0$). 

\begin{definition}[Gibbs measure in a finite box]  
Let $d\geq 2, N\geq 2$,  $0\leq \beta<\infty$ and $\Lambda \subset \Z^d$ be a finite box. Fix some boundary conditions (b.c.), $\bar \sigma \in(\S^{N-1})^{\p \Lambda}$. The \textbf{spin-$O(N)$ model} in $\Lambda$, with inverse temperature $\beta$ and b.c. $\bar \sigma$ is the following probability measure on $\sigma \in (\S^{N-1})^{\Lambda}$:
\begin{align*}  
d \FK{\Lambda,\beta}{\bar \sigma}{d\sigma} := \frac 1 {Z_{\Lambda, \beta}^{\bar \sigma}}
\exp\Big( - \frac \beta 2 \sum_{i\sim j \in \Lambda} \| \sigma_i - \sigma_j\|_{2, \R^N}^2 \Big) \prod_{i\in \Lambda \setminus \p \Lambda}  \lambda_{\S^{N-1}}(d\sigma_i) \prod_{g\in \p \Lambda} \delta_{\bar \sigma(g)}(d \sigma(g))\,.
\end{align*}
Here $\lambda_{\S^{N-1}}$ is the uniform measure on the sphere $\S^{N-1}$ and $Z_{\Lambda, \beta}^{\bar \sigma}$ is the finite normalisation constant (called partition function) to ensure that $\P_{\Lambda,\beta}^{\bar \sigma}$ is a probability measure. 

\bi
\item If $N=2$, this is a random collection of $\S^1$-valued spins called the $XY$ model. 
\item If $N=3$, this is a random collection of $\S^2$-valued spins called the classical Heisenberg model. 
\ei
\end{definition}

\begin{definition}[Infinite volume limits]  
If $d\geq 2, N\geq 2$ and $\beta \geq 0$ are fixed, an infinite volume limit of the spin-$O(N)$ model is any subsequential scaling limit of the above definition under a choice of increasing domains, for example $\Lambda_{k}
:=\{-k, \ldots, k\}^d \nearrow \Z^d$.  

By compactness of $\S^{N-1}$, it is easy to see that such subsequential limits exist. Especially when $N\geq 3$, the questions of convergence and uniqueness are known to be very challenging even when $d=2$ where uniqueness is conjectured to hold at all $T>0$. See for example the discussion in \cite{aru2025percolation}. 
\end{definition}

We now state our main result on the local structure of these Gibbs measures.

\begin{theorem}\label{th.SG}
Let $d=2$ and $N \geq 3$. 
There exist exponents $\kappa, \tau$, such that for all $\delta>0$, there exists an even  $L=L(\delta)  \asymp  \delta^{-\kappa}$ and $\beta(\delta) \asymp \delta^{-\tau}$ s.t for all $\beta \geq \beta(\delta)$ and uniformly over all boundary condition $\bar \sigma \in (\S^{N-1})^{\p \Lambda_L}$, one has 
\begin{align*}  
\FK{\Lambda_L, \beta}{\bar \sigma}{ \exists i\sim j  \in \Lambda_{L/4} \, \text{ s.t. } \|\sigma_i - \sigma_j\|_{2,\R^N} > \delta} < \delta \,.
\end{align*}
\end{theorem}

It has the following corollary on the structure of all infinite volume limits of the spin-$O(N)$ measure on $\Z^2$ (see Corollary \ref{c.SG2} in Section \ref{s.T} for a more precise and descriptive version). 

\begin{corollary}\label{c.SG}
Let $d=2$ and $N \geq 3$. 
For any $\delta>0$ and any $p\in (0,1)$, there exists $\bar \beta=\bar \beta(\delta,p)$ which is such that all $\beta$ Gibbs states for the spin-$O(N)$ model when $\beta \geq \bar \beta(\delta,p)$ a.s. satisfy the following properties
\bi
\item[i)] The density of edges $e$ with $\|\nabla \sigma(e)\| \leq \delta$ is a.s. higher than $p$. (N.B. The Gibbs state is not assumed to be translation invariant, so density here is understood as asymptotic density over growing boxes $\Lambda_r \nearrow \Z^2$).
\item[ii)] The set of low-gradient edges $\{ e, \|\nabla \sigma(e)\| \leq \delta \}$ is a.s. percolating on $\Z^2$.
\ei
\end{corollary}

\begin{remark}
Such a result is well-known for the Gibbs states of the spin-$O(N)$ model obtained as (subsequential) limits of the periodic Gibbs states on tori, thanks to the use of the so-called {\em chessboard estimate}, see \cite{frohlich1978phase,friedli2017statistical,peled2019lectures}. The interest of this Corollary is that the uniqueness of Gibbs states is out of reach at the moment without Griffith's second inequality (even under the assumption of translation invariance). 
\end{remark}

Finally, let us state a direct consequence of the proof of Theorem \ref{th.SG} which applies to Gibbs measures in dimensions $d\geq 3$.
\begin{proposition}\label{pr.NoSurface}
%

Let $d\geq 2$ and $N\geq 2$. There exists a universal constant $C = C(d,N)$ such that for any $n\geq 1$ and any $\alpha>0$, if $\beta$ is large enough (depending on $n$ and $\alpha$), then  uniformly over boundary condition $\bar \sigma \in  (\S^{N-1})^{\p \Lambda_n}$,  if $\sigma \sim \P_{\Lambda_n, \beta}^{\bar \sigma}$, then with probability at least $1-\alpha$, 
\begin{align}  
\text{DE}_{\Lambda_{n/2}}(\sigma) \le \begin{cases} C n^{d-2} & \text{if } N \ge 3 \\  C n^{d-2} \log^2(n) & \text{if } N = 2 \,.\end{cases} 
\end{align}
In particular, low temperature Gibbs states do not produce random surfaces with large gradients when $d\geq 3$. 
\end{proposition}


\ni
\textbf{Notations.}

\ni 
For any $d\geq 2$ and $n\geq 1$, $\Lambda_n:= \{-n, \ldots, n\}^d \subset \Z^d$. \\
For any finite set $\Lambda \subset \Z^d$, recall we define its boundary as 
$\p \Lambda := \{ i \in \Lambda \text{ s.t. } \exists j\in \Lambda^c \text{ with } i\sim j \}$.  \\
For any $A \subset \Lambda \subset \Z^d$ and $\psi: \Lambda \to \R^N$, \\
$\DE_A(\psi) := \frac 1 2 \sum_{e \in E(A)} \| \nabla_e \psi \|^2 = \sum_{i \sim j \in A} (\psi(i) - \psi(j))^2 $, where $E(A)= \{ \{i,j\} \mid i \in A, j \in A, i \sim j  \}$. We stress that we use the notation $i \sim j \in A$ for the sum over unoriented edges.

\medskip
\ni
{\bf Acknowledgments.} 
We wish to thank Fanghua Lin  for very inspiring discussions on harmonic maps in the continuum  setting. This adaptation to harmonic maps on a lattice $\Z^d$ would not have been possible for us without his clear and motivating explanations. We also wish to thank  Petru Mironescu, Christophe Sabot, Sylvia Serfaty and Daniel Stern for very useful discussions.

The authors acknowledge support from the ERC grant VORTEX 101043450. 

\medskip
\ni 
{\bf Statement of  AI use.}  We did not use any AI for the proofs of the main results: Theorem \ref{th.H}, Theorem \ref{th.DE} and Theorem \ref{th.SG}. The proof of Proposition \ref{pr.ExS1} was however entirely  found thanks to ChatGPT 6-Astra. We only rewrote it in our own words. As such, the use of AI is more substantial in this work than in our companion paper \cite{PerturbationS2}. 

Finally, AI was not used in the writing process of this paper, except the paragraph in the introduction explaining the different notions of harmonic maps which exist  in analysis.
Finally, this version of the  paper was not proofread using AI. 

%
%
%
%

\section{Hölder regularity for harmonic maps $\sigma : \Z^2 \to \S^2$}\label{sec:holder_reg}

The goal of this Section is to prove our main regularity estimates, i.e. Theorem \ref{th.H} and Theorem \ref{th.DE}.

%


\smallskip

The Section is organized as follows. In Subsection \ref{ss.DE}, we state two lemmas which control the Dirichlet energy of harmonic functions and harmonic maps. These Lemmas are stated in any dimension $d\geq 2$ and for any sphere-valued $\S^{N-1}, N \geq 2$. 
Then, in Subsection \ref{ss.ODE}, we rely on a discrete ODE to monitor how the Dirichlet energy of harmonic maps necessarily improves across scales. This will prove Theorem \ref{th.DE}. 
Finally in Subsection \ref{ss.HR}, we focus on the case $d=2$ and $N\geq 3$ and based on Theorem \ref{th.DE},  we conclude the proof of Hölder regularity of harmonic maps (Theorem \ref{th.H}). 

\subsection{Dirichlet energy estimates}\label{ss.DE}



\subsubsection{Extension of sphere valued functions}

In this section, we state two useful Lemmas which are partly inspired from their continuum analog. The proof of the first one, Lemma \ref{l.ext}, is quite different though since -- as opposed to the unit ball in $\R^d$ -- the heat-kernel of the simple random walk in cubic boxes of $\Z^d$ is somewhat less explicit.

\subsubsection{Bounds on the Dirichlet Energy of the harmonic extension in $\B^{N}$}

Our first Lemma controls the Dirichlet energy of the harmonic extension of a vector-valued boundary conditions inside a cubic box of $\Z^d$.  

In this Section, $\Lambda$ denotes any cubic box of $\Z^d$ (of the type $\Lambda_k = \{-k, \ldots, k\}^d$). 

If $\psi: \Lambda \to \R^{N}$ and $A \subset \Lambda $, recall from \eqref{e.DA} that we denote its Dirichlet energy measured inside the set $A$ by $\text{DE}_A = \frac 1 2 \sum_{e \in E(A)} |\psi_i - \psi_j|^2$. 

\begin{lemma}[Harmonic extensions of boundary conditions]\label{l.ext}
Let $\psi: \Lambda_n \to \R^{N}$ be  harmonic inside $\Lambda_n$.
Then 
$\text{DE}_{\Lambda_n}(\psi) \le C n^{(d-1)/2} \sqrt{\text{DE}_{\partial \Lambda_n }(\psi)} \|\psi\|_\infty$ for some 
$C$ that depends only on $N$ and $d$.\end{lemma}\label{lemma:harmonic_extension}

The proof of this Lemma can also be found in our paper \cite{PerturbationS2}. To make this note more self-contained, we include the main steps below and refer to \cite{PerturbationS2} for the few technical ones remaining.

\smallskip

{\em Main steps of the proof. (See \cite{PerturbationS2}).}

 As every coordinate of a harmonic function $\psi: \Lambda \to \R^{N}$ is harmonic, we only need to deal with real valued $\psi$.
Notice first that by harmonicity, if $x \in \Lambda^{\text{int}}$, \[ \psi(x) = \E_x ( \psi(S_\tau)) = \sum_{z \in \partial \Lambda} K(x,z) \psi(z)  \] where $(S_m)_{m \geq 0}$ is a simple random walk on $\Z^d$ and $\tau $ is the hitting time of the set $\partial \Lambda $ and $K(x,z)$ is the Poisson kernel of the Simple Random Walk (SRW) in $\Lambda$ defined by $K(x,z):=\P_x[S_\tau = z]$. 

Thus, by a discrete integration by part,  
\begin{align*}
\text{DE}(\psi) & = \sum_{x \in \partial \Lambda } - \psi(x) (\Delta \psi)(x) \\ 
&= \sum_{x \in \partial \Lambda } \sum_{y, y \sim x } \psi(x) (\psi(x )  - \psi(y))\\ 
&= \sum_{
x \in \partial \Lambda 
}
\left[\sum_{y \in \partial \Lambda , y \sim x } \psi(x) (\psi(x )  - \psi(y))\right] + \sum_{\substack{x \in \partial \Lambda \\ x \text{ is not a corner} }} \psi(x) (\psi(x )  - \psi(y_x)) 
\end{align*}
 where $y_x$ is the neighbor of $x$ in $\Lambda^{\text{int}}$. Notice that corners of $\partial \Lambda$ do not have any neighbors in $\Lambda^{int}$.

Note that, \margin{C: if this factor $\frac 1 2 $ is correct, the definition of $DE$ needs to be changed then ?, \green{N yes, done} \textbf{C: I don't think there is a factor 1/2 here ?  On the LHS, a point $x$ in $d=2$ will give $ 2* \psi(x)^2$} because it has $2$ points $y\sim x$, while the RHS, we will only have $2* \frac 1 2 \psi(x)^2$ because we have only two UNORDERED edges containing $x$. (We are summing over UNORDERED right?) Do you agree that we should remove the $\tfrac 1 2$ ? Do we need to change again the definition then ?}
\[ \sum_{ x \in \partial \Lambda} \sum_{y \in \partial \Lambda , y \sim x } \psi(x) (\psi(x )  - \psi(y))  = \frac{1}{2} \sum_{x, y \in \partial \Lambda, x \sim y} (\psi(x) - \psi(y))^2 \]  

and 

\[   (\psi(x )  - \psi(y_x)) = \sum_{z \in  \partial \Lambda } K(x,z) (\psi(x) - \psi(z))  \]

Thus, by another integration by part \margin{C: why ?? Another IBP with a kernel $K$ now ?\green{ Yes} \textbf{BUT MAYBE for the same reason $\tfrac 1 2$ is not there either ?}}
\[ \text{DE}(\psi ) = \text{DE}_{\partial \Lambda }(\psi) + \frac{1}{2} \sum_{x, z \in \partial \Lambda } K(x,z) (\psi_x- \psi_z)^2 ~. \]

Now it is standard that the Poisson kernel in such a cubic box satisfies $K(x,z) \le \frac{C}{|x -z|^d}$
(see for example \cite[Lemma D.1]{PerturbationS2}). Therefore,

\[ \sum_{x, z \in \partial \Lambda } K(x,z) (\psi_x- \psi_z)^2  \le C \sum_{x, z \in \partial \Lambda}  \frac{|\psi(x) - \psi(z)|^2}{|x -z|^d} \]

To finish the proof we use the following easy combinatorial  fact (see \cite[Lemma 3.7]{PerturbationS2} for a proof).
\begin{lemma} There exists a constant $C  = C(d)$  such that for any cubic box $\Lambda$, for all $k \ge 1$  and for any $x,y \in \partial \Lambda$ that are at distance $k$ in $\partial \Lambda$, one can choose some path  $\gamma_{x y}$ of minimal length inside $\partial \Lambda$ between $x$ and $y$ \margin{C: which stays inside $\p \Lambda$ ??, \green{yes}}, in such a way that every edge appears in at most $C k^{d-1}$ such paths.  \end{lemma}
\label{lemma:wellchosenpath}

We then write  $\psi(x) - \psi(z) = \sum_{e \in \gamma_{x z }} \nabla_e \psi $ and obtain

\begin{align*}
\sum_{x, z \in \partial \Lambda } K(x,z) (\psi_x- \psi_z)^2  &\le C \sum_{x, z \in \partial \Lambda}  \frac{|\sum_{e \in \gamma_{x z} } \nabla_e \psi|^2}{|x -z|^d}  \\ 
& = C \sum_{k  = 1 }^{ 2 n } \sum_{x \in \partial \Lambda} \left(\sum_{\substack{ z \in \partial \Lambda \\  |x - z | = k }} \frac{1}{k^d} \left| \sum_{e \in \gamma_{x z} } \nabla_e \psi \right|^2\right) 
\end{align*}

We then split the sum on $k$ at some well-chosen integer $A$ that we will specify.

By Cauchy--Schwarz inequality, 
\begin{align*}
\sum_{1 \le k  \le A  } \frac{1}{k^d}  \sum_{\substack{x, z \in \partial \Lambda \\  |x - z | = k }} \left|\sum_{e \in \gamma_{x z} } \nabla_e \psi \right|^2  & \le    \sum_{k  = 1 }^{ A } \frac{1}{k^{d}}  \sum_{\substack{ x,z \in \partial \Lambda \\  |x - z | = k }}  k \sum_{e \in \gamma_{x z} } (\nabla_e \psi )^2 \\ 
&\le C  A  \text{ DE}_{\partial \Lambda}(\psi) \,,
\end{align*}
where we have used Lemma \ref{lemma:wellchosenpath} to see that 

\[ \frac{1}{k^{d-1}}  \sum_{\substack{ x,z  \in \partial \Lambda \\  |x - z | = k }}  \sum_{e \in \gamma_{x z} } (\nabla_e \psi )^2 \le C \text{ DE}_{\partial \Lambda} (\psi)~. \]

In particular, if $A = n $, we see that 
\begin{equation}\label{eq:whenAisn} \sum_{1 \le k  \le A  } \frac{1}{k^d}  \sum_{\substack{x, z \in \partial \Lambda \\  |x - z | = k }} \left|\sum_{e \in \gamma_{x z} } \nabla_e \psi \right|^2  \le C  n  \text{ DE}_{\partial \Lambda}(\psi)  \end{equation}

For $k \ge A$, we simply bound $|\psi(x) - \psi(z)|  $ by $2\|\psi\|_\infty$ to obtain, 

\[ \sum_{|k| > A } \sum_{\substack{x , z \in \partial \Lambda, \\  d(x,z) = k }} \frac{1}{k^d} (\psi(x) - \psi(z ))^2 \le 2 n^{d-1} \|\psi\|_\infty^2 \sum_{|k| \ge A }  \frac{k^{d -2}}{k^d}  \le   \frac{2 C n^{d-1} \|\psi\|_\infty^2}{A} \]
(N.B. Here we used that $\Lambda=\Lambda_n$ and therefore there are $\Omega(1) n^{d-1}$ points $x$ on $\p \Lambda_n$).
We thus see that 
\[ \text{DE}(\psi) \le \text{DE}_{\partial \Lambda} (\psi) + C \, A \, \text{DE}_{\partial \Lambda}(\psi) + \frac{n^{d-1}}{2} C  \frac{\|\psi\|_\infty^2}{A} \] for some $C >0$. Optimizing on $A$ by setting $A = \frac{\|\psi\|_\infty n^{(d-1)/2}}{\sqrt{\text{DE}_{\partial \Lambda} (\psi)}}$ when $A \le n $, we obtain 

\[ \text{DE}(\psi) \le \text{DE}_{\partial \Lambda}(\psi) + C n^{(d-1)/2} \sqrt{ \text{DE}_{\partial \Lambda}(\psi) } \|\psi\|_\infty \]

The proof is complete as, by Cauchy--Schwarz inequality, one has when $\Lambda=\Lambda_n$
\[
\text{DE}_{\partial \Lambda}(\psi) \le n^{(d-1)/2}  \|\psi\|_\infty  \sqrt{\text{DE}_{\partial \Lambda}(\psi)}\,.
\]

\qed

\begin{remark}  
Notice that if we do not optimise on the cut-off value $1 \leq A \leq n$ in the above proof, it gives the following less optimal estimate which is also classical in harmonic analysis and harmonic maps (See for example (2.8) in \cite{lin2008analysis}): namely, there exists a constant $C>0$ s.t. for all $n\geq 1$, 
\begin{align}\label{e.LessOpt}
\text{DE}_{\Lambda_n}(\psi) \le C n \text{DE}_{\p \Lambda_n}(\psi)\,.
\end{align}

\end{remark}

\subsection{Projection to sphere-valued functions.}

Our second Lemma is strongly inspired by the analysis of Harmonic maps from \cite{hardt1988stable}.
Its proof is given in our work \cite[Lemma 3.8]{PerturbationS2}, we will give the main steps below for completeness.

\begin{lemma}\label{lemma:projection_on_sphere_trhough_a_good_point}
Let $D$ be any finite subset of a graph $G$. Let $\psi: D \to \R^{N}$ with $\psi(x) \in \mbb{S}^{N-1}  $  for any $x \in \partial D$ and $\|\psi(x)\| \le 1$ for any $x$ in $D$. Then, there exists $\sigma: D  \to \mbb{S}^{N-1} $ that coincides with $\psi$ on $\partial \Lambda$ and which satisfies 
 \[ \mrm{D E }(\sigma) \le \begin{cases} C_N \, \mrm{D E }(\psi) & \text{if } N\ge 3 \\ C_2 \mrm{D E }(\psi) \log(|D|) & \text{if } N= 2 \end{cases} \] for some constant $C_N > 0 $ that depends only on $N$ and where $|D|$ denotes the cardinal of $D$.
 \end{lemma}



\ni
{\em Proof of Lemma \ref{lemma:projection_on_sphere_trhough_a_good_point}. (See also\cite[Lemma 3.8]{PerturbationS2}).} 

 The proof relies on the following nice idea of Hardt, Kinderlehrer and Lin.

 A naive way to construct a function $\sigma: D \to \mbb{S}^{N-1} $ that coincides with $\psi$, is to set $\sigma(x) = \frac{\psi(x)}{|\psi(x)|}$. Of course, this is not necessarily well-defined and even when it is, it has no reason to have a Dirichlet energy that is under control. The choice of projection "from the point of view" of zero is somewhat arbitrary and one may wish to define a function $ \sigma_y (x) = \frac{\psi(x) - y}{|\psi(x) - y|} $ for some well-chosen $y$. The function $\sigma_y$ does not coincide with $\psi$ on  $\partial D$ with this definition but this can be circumvented as follows. For $y \in B(0,1)$, the function $f_y (x) = \frac{x - y }{|x -y |}$ is  bijective, bilipschitz on $\mbb{S}^{N-1}$ and the Lipschitz constant of $f_y$ and its inverse are bounded uniformly for $y \in B(0, 1/2)$.
 A suitable option should thus be to set
 $ \sigma(x) =f_y^{-1}(\sigma_y (x)) $
 for a well-chosen $y$.

What should then be the appropriate choice of $y$? The idea of Hardt, Kinderlehrer and Lin is very nice and simple: pick $y$ uniformly at random in $B(0,1/2)$ and hope it works with reasonable probability. Let $Y $ be a uniform random variable in $B(0,1/2)$. This strategy is indeed successful if $N \ge 3$ as one can see by computing the expectation of the  Dirichlet energy of $\sigma_Y$. (N.B. notice that since $f_Y$ is uniformly bilipschitz over $Y\in B(0,1/2)$, the Dirichlet energy of $x\mapsto f_Y^{-1}(\sigma_Y(x))$ is bounded from above by a constant times the Dirichlet energy of $x\mapsto \sigma_Y(x)$).

%
%
\ni
A short (but unpleasant) exact computation (carried in \cite[Lemma 3.8]{PerturbationS2}) gives for any given $Y\in B(0,\tfrac 1 2)$, 
  \begin{align*}
&|\sigma_Y (j) - \sigma_Y (j)| \le  |\psi(i) - \psi(j)|  \left( \frac{1}{|\psi(i) - Y|} + \frac{1}{|\psi(j) - Y|}\right) ~.  
\end{align*}


\ni
 By averaging over a uniformly chosen $Y\in B(0, \tfrac 1 2)$, we thus obtain 
 \[ \E(|\nabla_{i j } \sigma_Y |^2) \le |\nabla_{i j } \psi|^2  \frac{2}{\text{Vol}(B(0,1/2))} \int_{B(0,1/2)} \frac{ d r}{r^2}  \le C |\nabla_{i j } \psi|^2   \] if $N \ge 3$.

 Summing over edges $i j $, we see that 

 \[ \E(\text{DE}(\sigma_Y)) \le C \text{DE}(\psi) \] for some constant $C$ that depends only on $N$.

 We have crucially used the fact that $ \frac{1 }{r^2}$ is integrable near $0$ which only holds when $N \ge 3$. If $N = 2$, we need to slightly modify the distribution of $Y$ in order to solve the problem. Let $\tilde{Y}$ be a random variable which has the distribution of a  uniform random variable on  $B(0,1)$ conditioned on not lying in any of the $|D|$ balls of radius $\frac{1}{10 |D|}$ centered at the $\psi_i$, $i \in D$.

We obtain,
  \[ \E(|\nabla_{i j } \sigma_{\tilde{Y}} |^2) \le |\nabla_{i j } \psi|^2  \frac{2}{C_\psi} \int_{A_{(\frac{1}{10 D }, 1/2) }} \frac{ d r}{r^2} \]
  
where $C_\psi = \text{Vol}\left(B(0,1/2) \setminus \bigcup_{i \in D} B(\psi_i, \frac{1}{10 |D|}) \right)  \ge C $ for some $C>0$ and $A_{(\frac{1}{10 D }, 1/2) } $ is the annulus centered at $0$ with inner radius $\frac{1}{10 |D|}$ and outer radius $1/2$.

As (when $N=2$), \[ \int_{A_{(\frac{1}{10 D }, 1/2) }} \frac{ d r}{r^2} \le C \log(|D|) ~, \] we see that  $ \E(|\nabla_{i j } \sigma_{\tilde{Y}} |^2)  \le C (\nabla_{i j }(\psi ))^2 \log(|D|)$ for some other constant $C$ and 

\[ \E(\text{DE}(\sigma_{\tilde{Y}} (x))) \le C \text{DE}(\psi) \log(|D|) ~. \]
\qed

\subsection{Proof of Theorem \ref{th.DE} via a discrete ODE}\label{ss.ODE}
$ $

The goal of this Section is to prove the Dirichlet Energy estimates in the bulk $\Lambda_{n/2}$ uniformly over boundary conditions as stated in Theorem \ref{th.DE}. 

%
%
%
%

Before proving Theorem \ref{th.DE}, let us first state an immediate Corollary. We state it here because its continuum analog plays a crucial role in the regularity theory of harmonic maps from $\R^2$ to $\S^2$. See for example \cite[Chapter 2]{lin2008analysis}. We shall not need such a  ``small enough $\eps$-regularity  statement in this paper'' and the proofs below will bypass this classical result.


\begin{corollary}[$\epsilon$-Regularity Theorem for lattice based harmonic maps]\label{pr.Eps}
When $d=2$ and $N\geq 3$, for any $\eps>0$, there exists $\lambda=\lambda(\eps) \in (0,1)$ so that for any $n\geq 10$, then uniformly in all possible boundary conditions $\bar \sigma  \in (\S^2)^{\p \Lambda_n}$, any harmonic map $\sigma$ with boundary data $\bar \sigma$ needs to accumulate a Dirichlet energy inside the bulk box $\Lambda_{\floor{\lambda n}} \subset \Lambda_n$ which is less than $\eps$. I.e. 
\begin{align*}  
\sum_{i\sim j \in \Lambda_{\floor{\lambda n}}} \|\sigma^{opt}_i - \sigma^{opt}_j\|^2 \leq \eps
\end{align*}
\end{corollary}

\ni
{\em Proof of  Theorem \ref{th.DE}.} 

The starting point of this step is to note that Lemma \ref{l.ext} together with Lemma \ref{lemma:projection_on_sphere_trhough_a_good_point} imply that if  $\sigma: \Lambda \to \mbb{S}^{N-1}$ is  a harmonic map on $\Lambda$ (it is important here that $\Lambda$ is a cubic box $\Lambda=\Lambda_k \subset \Z^d$), we have 
\begin{equation} \label{eq:upperbound} \text{DE}(\sigma) \le  \begin{cases}  C  n^{(d-1)/2} \sqrt{\text{DE}_{\partial \Lambda }(\sigma)}  & \text{if } N\ge 3 \\ C \log(n)  n^{(d-1)/2} \sqrt{\text{DE}_{\partial \Lambda }(\sigma)}  & \text{if } N =  2 \end{cases} ~. \end{equation}

We use this information to derive a discrete differential inequality on $H(k) :=  \text{DE}_{\Lambda_k}(\sigma)$.

As, 
\[ H(k) - H(k - 1) = \sum_{e \in \partial \Lambda_k } (\nabla_e \sigma)^2 + \sum_{\substack{e = \{x,y\}, \\  x \in \partial \Lambda_{k-1}, y \in \p \Lambda_{k} }} (\nabla_e \sigma)^2  \ge \text{DE }_{\partial \Lambda_k}(\sigma) \]

we see that

\[ H(k ) - H(k-1) \ge \frac{H(k)^2}{k^{d-1}}  \begin{cases} C & \text{if } N \ge 3 \\ \frac{C}{\log^2(k)} & \text{if } N = 2 \end{cases} ~. \]

By replacing $H(k)^2  $ by $H(k) H(k-1)$ in the right-hand side of this last inequality, we obtain  $\frac{1}{H(k-1)} - \frac{1}{H(k)} \ge \frac{C}{k^{d-1} } $ if $N\ge 3$ and the same expression with an extra $\log^2(k)$ term in the denominator of the right-hand side if $N = 2 $.

Summing from $L$ to $n$ we get, 

\[ \frac{1}{H(L)} \ge \frac{1}{H(L)} - \frac{1}{H(n)}  \ge \sum_{k = L }^n \frac{1}{k^{d-1}} \] if $N \ge 3 $ and an extra $\log^2(k)$ in the denominator of the summand if $N  = 2$.

If $d \ge 3$, it holds that $ \sum_{k = n/2 }^n \frac{1}{k^{d-1}} \ge  C(d) n^{2 - d }   $ and

\[ \sum_{k = n/2 }^n \frac{1}{k^{d-1} \log^2(k)} \ge  C(d) n^{2-d} \log^{-2}(n)   \]

and we get the desired result.

When $d = 2$ and $N \ge 3$, we obtain  

\[ H(L) \le C \frac{1}{\log(n/L)} \] leading to both the large scale and small scale result.

When $N = 2 $ and $d = 2 $, 
\[ H(L) \le \frac{C}{\sum_{k = L }^n \frac{1}{k \log^2(k)}  } \le  \frac{C}{ \frac{1}{\log(L)} - \frac{1}{\log(n)}} \] which reads 

\[ H(n/2) \le C \log^2(n) \] and gives no information at small scales.   \qed

\subsection{Proof of Hölder regularity when $d=2$ and $N\geq 3$.}\label{ss.HR}
$ $

In this Section, we conclude the proof of Theorem \ref{th.H} which applies to $d=2$ and $N\geq 3$. Notice that in this case, Theorem \ref{th.DE} already gives a weaker ``local logarithmic regularity''. 
The idea of the proof  is inspired by its continuum analog (as in \cite[Chapter 2]{lin2008analysis}). Let us start with the following particular case. 

\smallskip

\ni
\textbf{The case of nearest-neighbor points $x\sim y$.}

 Let us fix two points $x,y \in \Lambda_n$ and let $N\geq 1$ such that $x + \Lambda_N \subset \Lambda_n$. Let $\sigma$ be one harmonic map extension of the boundary condition $\bar \sigma$ along $\p \Lambda_n$. Our goal is then to show that $\|\sigma(x) - \sigma(y)\|_2 \leq C n^{-\alpha}$. 
 The proof will be divided in 2 steps. 
\bnum
\item First, thanks to Theorem \ref{th.DE} we find that $\text{DE}_{x+ \Lambda_{\frac n 2}}(\sigma) \leq C(2,N)$ (here, as in Theorem \ref{th.H}, $\sigma$ is one of the possible harmonic maps with boundary conditions $\bar \sigma$).


Note that one could use the second part of Theorem \ref{th.DE} to deduce that for scales  $r_0 \ll \frac N 2$, 
\begin{align*}  
\text{DE}_{x+ \Lambda_{r_0} }(\sigma) \le \frac{C}{\log(n/r_0)} 
\end{align*}
but such a logarithmic decay is not good enough to imply Hölder regularity.  Instead, starting at scale $\frac n 2$, we change gear and use the following idea. 

\item From then on, (scale $\frac n 2$ and below), we show that the Dirichlet Energy improves at each multiplicative scale. This is the purpose of the next Lemma.
\enum


\begin{lemma}\label{l.dyadic}
There exists $\lambda \in(0,1)$ so that for any $R$ s.t. $ 10 \leq R \leq n$, if 
if $\sigma$ is a harmonic map on $\Lambda_n \supset \Lambda_R$ then 
\begin{align*}  
\text{DE}_{\Lambda_{\floor{R/2}} }(\sigma) \leq 
\lambda \, \text{DE}_{\Lambda_{R}}(\sigma)\,.
\end{align*}
\end{lemma}

Before proving this Lemma, let us shortly explain why steps (1) and (2) imply the local Hölder regularity at $x\sim y$. 
Let $k\geq 0$ be such that $\frac N 4 < 2^{k+1} \leq \frac N 2$. By applying the above Lemma $k$ times, we obtain that the Dirichlet energy of $x+ \Lambda_1$ (which contains the edge $x\sim y$) is less than $\lambda^k  \DE_{x+\Lambda_{2^{k+1}}}(\sigma) \leq \lambda^k \DE_{\Lambda_{N/2}} \leq \lambda^k C(2,N) \leq C(2,N) N^{-2\alpha}$ where $\alpha=\alpha(\lambda)>0$. Since $\|\sigma(x) - \sigma(y)\|_2^2$ is less than $\DE_{x+\Lambda_1}(\sigma)$, this gives the result.


\smallskip
\smallskip
\smallskip

%

\ni
{\em Proof of Lemma \ref{l.dyadic}.}

Let us assume that $\sigma$ is a harmonic map on $\Lambda_R$.
Obviously, we must have  \margin{C: formally I should insert many $\floor{\frac R 2}$ and $\floor{\frac R 2} + 1$... }
\begin{align*}  
\sum_{i=\frac R 2}^R  \text{DE}_{\p \Lambda_{i}}(\sigma) \leq \text{DE}_{\Lambda_{R} \setminus \Lambda_{R/2}}(\sigma)\,.
\end{align*}
There must also exist $j\in \{ \floor{\frac R 2}, \ldots, R\}$ such that 
\begin{align*}  
\text{DE}_{\p \Lambda_{j}}(\sigma) \leq  \frac{3 } R  \text{DE}_{\Lambda_{R} \setminus \Lambda_{R/2}}(\sigma) \,.
\end{align*}
Let $\psi=\psi_j$ be the harmonic extension in $\Lambda_j^{int}$ of $\sigma_{| \p \Lambda_j}$. 
If we apply estimate \eqref{e.LessOpt} to this harmonic extension $\psi$, we obtain
\begin{align*}  
\text{DE}_{\Lambda_j}(\psi) \le C_1 j \text{DE}_{\p \Lambda_j}(\psi)  \leq  C_1 \,j \,   \frac{3} R\, \text{DE}_{\Lambda_{R} \setminus \Lambda_{R/2}}(\sigma) \,,
\end{align*}
for some universal constant $C_1=C_1(N)$. 

Now using Lemma \ref{lemma:projection_on_sphere_trhough_a_good_point}, we now obtain 
\begin{align*}  
\text{DE}_{\Lambda_j}(\sigma) \leq C_2(N) \text{DE}_{\Lambda_j}(\psi)  \leq C_1(N) C_2(N)  \,  j \, \frac{3 } R\, \text{DE}_{\Lambda_{R}\setminus \Lambda_{R/2}}(\sigma) \,.
\end{align*}
Now since $\text{DE}_{\Lambda_{R/2}}(\sigma) \leq \text{DE}_{\Lambda_j}(\sigma)$ and since $\frac j R \leq 1$, we obtain that 
\begin{align*}  
\text{DE}_{\Lambda_{R/2}}(\sigma) & \leq  3  C_1(N) C_2(N) \text{DE}_{\Lambda \setminus \Lambda_{R/2}}(\sigma)  \\
& \le 3  C_1(N) C_2(N)  \Big[ \text{DE}_{\Lambda_R}(\sigma)  - \text{DE}_{\Lambda_{R/2}}(\sigma)  \Big]
\end{align*}
which thus proves the result with   \margin{C: je n'ai pas fait gaffe aux coupures entre $\Lambda_{R/2}$ et $\Lambda_R \setminus \Lambda_{R/2}$ ni aux $\floor{R/2}$...  \textbf{pas sur ceci dit qu'on veuille alourdir le texte ?} \dgreen{N: moi ça ne me choque pas mais si tu penses que c est nécessaire on peut tout préciser}} 
\begin{align*}  
\lambda:=\frac{3  C_1(N) C_2(N) } {1 + 3  C_1(N) C_2(N) }\,. 
\end{align*}
\qed

\ni
\textbf{The case of arbitrary points $x, y \in \Lambda_n$.}

Let us assume we can find $r\geq 2$ such that $y\notin x+ \Lambda_{r/4}$ and $y\in x+ \Lambda_{r/2}  \subset  x+ \Lambda_{2r} \subset \Lambda_{n}$. If this is not the case, the ratio $\frac{d(x,y) }{d(x, \partial \Lambda_n)\wedge d(y, \partial \Lambda_n)} $ is of order $1$. Notice that $r$ is of order $\dist(x,y)$.

\smallskip


\ni
\textbf{Step 1.} Let $j_0$ be such that $x + \Lambda_{2^{j_0} r } \subset \Lambda_n$ and $x + \Lambda_{2^{j_0+1} r }$  is not included in  $\Lambda_n$. (Note that by our assumption $x+\Lambda_{2r} \subset \Lambda_n$, we have $j_0\geq 1$) Notice that $ \frac{1}{2^{j_0}}$ is of the same order as $\frac{d(x,y) }{d(x, \partial \Lambda_n)\wedge d(y, \partial \Lambda_n)} $.
Let us consider all dyadic boxes ranging between $x+ \Lambda_r$ and $\Lambda_n$, i.e. $x+ \Lambda_{2^j r}$, $j=1, \ldots, j_0 $ and let $L = 2^{j_0} r$.  Thanks to Theorem \ref{th.DE}, going from $x+\Lambda_{2^{j_0} r}$ down to $x+ \Lambda_{2^{j_0-1} r}$,  we have that $\DE_{x+ \Lambda_{2^{j_0-1} r}}(\sigma) \leq C(2,N)$. Then, using Lemma \ref{l.dyadic} as we did above, we get 
\begin{align}\label{e.lr}
\DE_{x+\Lambda_r} \leq O(1) \Big( \frac{r}{L} \Big)^{2\alpha}~.
\end{align}
\margingreen{ I rewrote a lot of this step}

\smallskip
\ni
\textbf{Step 2.}

Let us introduce some notations of discrete differential calculus over a graph. If $G = (V,E)$ is a (unoriented) graph, we define $\overrightarrow E$ as being the set of oriented edges (each edge appear in both orientation). If $e \in \overrightarrow{E} $ we denote $e_+$ and $e_-$ the beginning and the end of the edge. If $e = (e_+, e_-)$, we denote by $-e$ the edge $(e_-, e_+)$. We say that  $\gamma:\overrightarrow{E}\to \R$ is a $1$-form if $\gamma(e) = - \gamma(-e)$ for all $e \in \overrightarrow{E}$. For $f: V \to \R$ and $e \in \overrightarrow{E}$, we denote $d f (e) = f(e_+) - f(e_-)$ for the differential of the function $f$. The function $d f$ is a $1$-form. If $\gamma$ is a $1$-form we define the divergence of $\gamma$ by $\div \gamma(x) = \sum_{e \in \overrightarrow{E}, x = e_+} \gamma(e)$. For vertices, $x , y \in V$, we say that a $1$-form $\gamma$ is a flow on between $x$ and $y$ if $ \div(\gamma)(x) +  \div(\gamma)(y) = 0$ and $\div \gamma(u) =  0$ for any $u \notin \{x, y\}$. \margingreen{J'ai rajouté ce paragraphe, qu'en penses tu?}

We claim that there exists a flow $u$ on the directed edges of $x+\Lambda_r$, which is such that 
\bnum
\item $\mathrm{div}(u)  =  \delta_x - \delta_y$ so that $\<{d \sigma, u} = \sigma(x) - \sigma(y)$. 
\item For any edge $u=(i,j), i\sim j$ inside $x+ \Lambda_r$, we have 
\begin{align*}  
|u(e)| \leq C \Big( \frac 1{\dist(i,x)} + \frac 1 {\dist(i,y)} \Big)\,.
\end{align*}
\ni
\enum
This is proved for example in \cite[Lemma 3]{sabot2013} (stated in dimension $d\geq 3$ but the same proof applies to $d=2$). 

\smallskip
\ni
\textbf{Step 3.} We now partition the box $x+\Lambda_r$ into successive dyadic boxes around both $x$ and $y$ as follows. 
\begin{align*}  
x+\Lambda_r = \bigcup_{k=-1}^{\log_2(\frac r {8})} (x+ \Lambda_{2^{k+1}})\setminus (x+\Lambda_{2^k}) \bigcup 
(y+ \Lambda_{2^{k+1}})\setminus (y+\Lambda_{2^k})  \bigcup A_{x,y,r}\,,
\end{align*}
where $A_{x,y,r}$ stands for all the points which remain outside the dyadic annuli. (Also we define for convenience $\Lambda_{1/2}:= \emptyset$).

Under this decomposition, by using the triangle inequality and then Cauchy-Schwarz on $\sigma(x)-\sigma(y)=\sum_{e\in \vec{E}(x+\Lambda_r)} u(e) \nabla \sigma(e)$ we find, 
\begin{align*}  
\| \sigma(x) - \sigma(y)\|_2 \leq & O(1) \sum_{k=-1}^{\log_2(\frac r {8})}  \sqrt{2^{2k} * \frac 1 {2^{2k}}}
\Big( \sqrt{\DE_{x+ \Lambda_{2^{k+1}}}(\sigma)} + \sqrt{\DE_{y+ \Lambda_{2^{k+1}}}(\sigma)} \Big) \\
& + \sqrt{r^2 \frac 1 {r^2}}
O(1) \sqrt{\DE_{x+ \Lambda_{r}}}\,.
\end{align*}
(N.B. The factor $\frac 1 {2^{2k}}$ comes from the estimate $|u|^2(e) \leq 2^{-2k}$ in the annulus at distance $2^k$ from $x$ or $y$). 
Now, using again Lemma \ref{l.dyadic} as we did above together with estimate \eqref{e.lr}, we have for each $k$, 
\begin{align*}  
\DE_{x+ \Lambda_{2^{k+1}}}(\sigma) \vee 
\DE_{y+ \Lambda_{2^{k+1}}}(\sigma) \leq C (\frac {2^{k+1}} L)^{2\alpha}\,,
\end{align*}
which implies 
\begin{align*}  
\| \sigma(x) - \sigma(y)\|_2 \leq O(1) \sum_{k=0}^{\log_2(\frac r {4})}  
\Big( \frac {2^k} {L}\Big)^\alpha + O(1) \Big( \frac {r} {L}\Big)^\alpha\,,
\end{align*}
which thus ends the proof.  \qed

\begin{remark}
It is well known that once Hölder regularity is obtained, it is in general easier to improve it to higher notions of regularity such as ``Discrete-smooth'' regularity (by using more classical analysis tools such as Schauder-type estimates).
 See for example the reference  \cite{thomas2001regularity} which implements this under the extra assumption that $\sigma$ stays inside an open set of the half-sphere $\S^2_+$. This is a rather restrictive assumption in general but our Hölder regularity result Theorem \ref{th.H} implies it must indeed be the case in the bulk for harmonic maps in large domains. We thus refer to  Section 3.2 of \cite{thomas2001regularity} combined with our Theorem \ref{th.H} to obtain a greater regularity. 
 \end{remark}

%
%
%

\section{Turning on the temperature}\label{s.T}

\subsection{Probabilistic description of Gibbs states when $d=2$ and $N\geq 3$.}

$ $



Recall that our main {\em small gradient} estimate, Theorem \ref{th.SG}, is stated for Gibbs measures in finite domains and holds uniformly on boundary conditions. 
Using the very useful stochastic domination with product measure from \cite{liggett_domination_1997}, it readily implies the following stronger version of Corollary \ref{c.SG}.  It gives a quantitative control on the structure of all Gibbs measure of the spin-$O(N)$ model in dimension $d=2$. 

\begin{corollary}\label{c.SG2}
Let $d=2$ and $N \geq 3$. 
For any $\delta>0$ and any $p\in (0,1)$, there exists $L_2=L_2(\delta) \,\, (\geq \Omega(1) \delta^{-\kappa})$ and
$\beta_2=\beta_2(\delta,p)$ \margin{C: we can make sure $L_2$ does not depend on $p$, only some margin with the groundstate will suffice here. The temperature on the other hand will be affected by $p$. \dgreen{N: I agree}} such that  if $\beta \geq \beta_2(\delta,p)$ then any infinite-volume Gibbs measure
\footnote{It is conjectured that there is a unique such Gibbs measure at any fixed $\beta<\infty$, but this is still open, to our knowledge.} of the spin-$O(N)$ model on $\Z^2$ satisfies the following coarse-grained stochastic domination property.

Let $\mathrm{Grid}_{L_2}:= \{ Q_t \}_{t \in T}$ be a partition of $\Z^2$ into blocks which are translates of $\Lambda_{L_2} ( = \{- L_2, \ldots, L_2\}^2 \subset \Z^2)$. $T$ here is any choice of countable set to index this partition (for example $\Z^2$). 

If $\sigma$ is sampled according to an infinite volume limit of the spin-$O(N)$ model on $\Z^2$, we associate to $\sigma$ the following site-percolation $\omega=\omega(\sigma)$ indexed by $T$: for every $t\in T$, we let 
\begin{align*}  
\begin{cases}
\omega_t = \omega(\sigma)_t := 1 & \text{  if  for all $i\sim j \in Q_t$ } \|\sigma_i -\sigma_j\|_{2,\R^N} \leq \delta. \\
\omega_t = \omega(\sigma)_t := 0 & \text{  otherwise }\,.
\end{cases}
\end{align*}
\margin{C: or maybe or more global condition as in Theorem \ref{th.H} ? \dgreen{I believe this is heavy enough of a statement?}}
We claim that under the above assumptions (most importantly if $\beta \geq \beta_2(\delta,p)$), then the site configuration $\omega$ stochastically dominates a $p$-Bernoulli site percolation on $T$, i.e. 
\begin{align*}  
\{ \omega(\sigma)_t, t\in T\} \succcurlyeq (p \delta_1 + (1-p)\delta_0)^{\otimes T}\,.
\end{align*}
(Notice that we do not know a priori that the left-hand side $\{ \omega(\sigma)_t, t\in T\}$ is translation invariant, while the right-hand site is). 

This means in a very quantitative way that most edges in low temperature spin-$O(N)$ models have small gradients. 
\end{corollary}


Before the proof of Theorem \ref{th.SG}, let us make a few remarks. 
\begin{remark}  
We believe (but do not prove) that this statement of Theorem \ref{th.SG} is wrong in the case of $N=2$ !
Indeed, if we equip the boundary of a large domain $\Lambda_L$ with a boundary condition coming from a non-trivial harmonic map that has a vortex at the origin as in Proposition  \ref{pr.ExS1}, one of the associated ground states produces a vortex around the origin of the graph $\Lambda_L$. (See Figure \ref{f.XY}. It does not exactly correspond to this situation but it gives an idea.) See for example the references \cite{bethuel1994ginzburg,sandier2007vortices,pacard2012linear} for the Ginzburg--Landau model in the continuum. \margingreen{N:I don't know if we want to cite this here}

Once the temperature is turned on but is sufficiently small (i.e. $T=T(L)$ small enough and $\beta=\beta(L)=T(L)^{-1}$ large enough), the Gibbs measure should not fluctuate much around this ground state. See Open Problem \ref{op.vortex}.
\margin{C: maybe this is now under reach ?? By a slight adaptation of the proof of non-trivial $\S^1$ harmonic maps ? still something to be done though $:)$ -> N: One would need to be able to show uniqueness of minimizer (uniqueness of the position of the vortex) and much more on this minimizer, it does not seem simple to me }
\end{remark}


%

\begin{remark}  
On the other hand, we believe Corollaries \ref{c.SG} and \ref{c.SG2} may in fact also apply to the $2d$ XY model, despite the lack of Theorem \ref{th.SG}. See Open Problem \ref{op.perco}. 
\end{remark}

\subsection{Proofs in the case $d=2$. }

In order to prove Theorem \ref{th.SG}, it will be sufficient to obtain a good enough upper bound on each fixed edge $i \sim j \in \Lambda_{L/2}$ and then to conclude by union bound. (N.B. To obtain a slightly better estimate, still polynomial in $\delta$ though, we could organise the union bound in a more collective way using a suitable geometric dyadic structure of scales).   

To obtain the bound for one edge, we shall use repeatedly the following scheme:
\begin{align*}  
&\text{Markov property at $\p \Lambda_R$} \\
\, \to \,&\text{Upper bound on interior Dirichlet Energy at $T=0$ (harmonic maps)}  \\
 \, \to \,& \text{Upper bound on interior Dirichlet Energy at $T>0$ with high probability} \\
\, \to \,& \text{Markov property at $\p \Lambda_{R-1}$} \\
\, \to \, & \ldots  \\
 \, \Rightarrow \, & \text{ ODE satisfied with high probability across scales.}
\end{align*}

%

\begin{lemma}\label{l.T}
For any finite $\Lambda \subset \Z^2$ and any boundary condition $\bar \sigma \in (\S^{N-1})^{\p \Lambda}$, let 
\begin{align*}  
GS_{\Lambda}(\bar \sigma): = \min_{u\in (\S^{N-1})^{\Z^2}, u_{\md \p \Lambda}=\bar \sigma} \,\,\, 
\DE_\Lambda(u).
\end{align*}
Then for any $\gamma, \alpha>0$, there exists $\beta(\Lambda,\gamma, \alpha)>0$ such that for all $\beta \geq \beta(\Lambda, \gamma, \alpha)$, one has
\begin{align*}  
\FK{\Lambda, \beta}{\bar \sigma}{ \DE_\Lambda(\sigma)>  4\, GS_{\Lambda}(\bar \sigma)  + \gamma^2 } < \alpha \,.
\end{align*}

\end{lemma}
\begin{remark}
   will discuss in Subsection \ref{ss.quant} how to make $\beta(\Lambda,\gamma,\alpha)$ quantitative. 
\end{remark}

\ni
{\em Proof of Lemma \ref{l.T}.}

Let $\sigma^0$ be one of the ground state configurations on $\Lambda$ that coincides with $\bar \sigma$ on $\partial \Lambda$. I.e. $\DE_\Lambda(\sigma^0) = GS_\Lambda(\bar \sigma)$.  Let $E(\Lambda)$ be the number of edges $| \{ i \sim j, i \in \Lambda\} |$. 

Any configuration $\sigma$ such that for all $i\in \Lambda$, $ |\sigma_i - \sigma^0_i|^2 \le \frac{ \gamma^2 }{ 16 E(\Lambda)}$  satisfies \footnote{In order to improve the quantitative dependence of $\beta$ on $\delta$, one could improve the analysis here by making the neighborhood of $\sigma$ larger at each $\sigma_i$, using instead $ |\sigma_i - \sigma^0_i|^2 \le \frac{  3GS_\Lambda(\bar \sigma) + \gamma^2 }{ 16 E(\Lambda)}$. For simplicity, we do not follow this path here.}
\begin{align*}  
\DE_\Lambda(\sigma)  & = \frac 1 2  \sum_{i\sim j \in \Lambda} \| \sigma_i -\sigma_j\|^2  \\
& \leq \frac 3 2  \left(  \sum_{i\sim j  \in \Lambda}  \| \sigma_i - \sigma_i^0\|^2 +  \| \sigma^0_i - \sigma_j^0\|^2 + \| \sigma_j - \sigma_j^0\|^2\right) \\
& \leq  3 GS_{\Lambda}(\bar \sigma)  +  \frac 6 {16}  \gamma^2  \\
& \leq 3 GS_{\Lambda}(\bar \sigma) + \frac 1 2 \gamma^2\,.
\end{align*}

In particular, the set of such $\sigma$ with $\DE_{\Lambda}(\sigma) \le  3 GS_{\Lambda}(\bar \sigma) + \frac 1 2 \gamma^2$ has a positive Lebesgue  measure that is bounded from below by 
\begin{align*}  
C(\Lambda, \gamma) := \left\{  c(N) \Big(\frac{  \gamma^2 }{ 16 E(\Lambda)} \wedge 1 \Big)^{\frac  {N-1} 2} \right\}^{|\Lambda|}
\end{align*}

This gives us the following lower bound on the partition function namely
\begin{align*}  
Z_{\Lambda,\beta}^{\bar \sigma} \geq C(\Lambda, \gamma)  \exp\Big( - \beta  (3 GS_{\Lambda}(\bar \sigma) + \frac 1 2 \gamma^2) \Big).
\end{align*}

Thus,
\begin{align*}
\FK{\Lambda, \beta}{\bar \sigma}{ \DE_{\Lambda}(\sigma) >  4\, GS_{\Lambda}(\bar \sigma)  + \gamma^2 }  & = 
\frac 1 {Z_{\Lambda,\beta}^{\bar \sigma}} 
\int_{\sigma , \DE_{\Lambda}(\sigma) >  4\, GS_{\Lambda}(\bar \sigma)  + \gamma^2  } \exp\Big( - \beta \DE^{\Lambda}(\sigma) \Big) d \sigma  \\ 
&   \le \frac{
\exp\Big( - \beta (4\, GS_{\Lambda}(\bar \sigma)  + \gamma^2)\Big)
}{
Z_{\Lambda,\beta}^{\bar \sigma}
} \\
&\le \frac{\exp\Big(- \beta ( GS_{\Lambda}(\bar \sigma)  + \frac 1 2 \gamma^2) \Big)}
{
C(\Lambda,  \gamma)  
}\le  \frac{\exp\Big(- \frac 1 2 \beta \gamma^2 \Big)}
{
C(\Lambda,  \gamma) }
\end{align*}
which can be made smaller than $\alpha$ by choosing $\beta$ sufficiently large. 
\margin{\bf C:  when we will need further below -- subsection \ref{ss.quant} -- to make this quantitative (in $\gamma, GS_\Lambda(\bar \sigma)$ etc.) this will require a bit of work, but I feel it is better to do it a posteriori than all at once otherwise it would be impossible to read and understand the crude bound. What do you think ?\dgreen{I agree}}

\qed

\ni
{\em Proof of Theorem \ref{th.SG}.}

We shall fix $L=L(\delta)$ and $\beta=\beta(\delta)$ sufficiently large further below. 

We first follow the ODE proof of Theorem \ref{th.DE} (from subsection \ref{ss.ODE}) in order to show the following estimate.
\begin{lemma}\label{l.odeT}
There exists a constant $C=C(N)$ which does not depend on $\delta$ nor on the scale $L(\delta)$ such that the following holds: for any $\alpha>0$,  if $\beta$ is chosen sufficiently large (depending only on $\alpha$ and $L(\delta)$), then uniformly over  $\bar \sigma$,
\begin{align*}  
\FK{\Lambda_L,\beta}{\bar \sigma}{\DE_{\Lambda_{L/2}}(\sigma) \leq C(N)} \geq 1 -\alpha\,.
\end{align*}
\end{lemma}

The proof of this Lemma follows the same ODE as for Theorem \ref{th.DE} (in Subsection \ref{ss.ODE}) except we need to deal with some additional disorder. To handle this disorder, we will just prove that if $\beta$ is large enough then by loosing some multiplicative constant, we may stay as close as we want to the ODE from scale $k$ down to scale $k-1$. 

Our proof will explore in a  Markovian way first the slice $\sigma_{\p \Lambda_{L}}$ and then $\sigma_{\p \Lambda_{L-1}}$ and so on and so forth. That is to say that, when analysing the difference between the energy at scale $k$ and at scale $k-1$, we apply the Markov Property at scale $k$ around $\p \Lambda_k$. 

We also need to introduce the following stopping time: let $N$ be the first scale along the inward Markovian exploration procedure defined by 
\begin{align*}  
U:= \sup \Big\{\frac L 2 \leq k \leq L  \text{ s.t. } GS_{\Lambda_k}(\sigma_{\md \p \Lambda_k}) \leq 1  \Big\}\,. 
\end{align*}
(We define $U:=0$ if there is no such maximal radius).


For any $\frac L 2  \leq k \leq L$, let $H(k)$ be the random variable (under the measure $\P_{\Lambda,\beta}^{\bar \sigma}$)
\begin{align*}  
H(k):= \DE_{\Lambda_k}(\sigma), \,\, \sigma \sim \P_{\Lambda,\beta}^{\bar \sigma}\,.
\end{align*}

%
%
%

Let $k$ be any scale $\frac L 2  \leq U \leq k \leq L$. 
If $k>U$, then choose $\gamma^2$ to be equal to $1$ in Lemma \ref{l.T} and consider the event 
\begin{align*}  
A_k:=\left\{  H(k) \leq   5\, GS_{\Lambda_k}(\sigma)  \right\}~.
\end{align*}
We have that 
\begin{align*}  
\FK{\Lambda_k,\beta}{\sigma_{\p \Lambda_k}}{A_k}\geq 1-\alpha\,. 
\end{align*}
If instead $k$ is equal to the stopping radius $U$, then we have no control on how small $GS_{\Lambda_{U=k}}(\sigma_{\p \Lambda_U})$ is, but taking $\gamma^2=1$, notice we still have by applying Markov property at that scale that
\begin{align*}  
\FK{\Lambda_U,\beta}{\sigma_{\p \Lambda_U}}{
H(U) \leq   5  }\geq 1-\alpha\,. 
\end{align*}
As such, if we ever reach stopping radius $U$, by choosing any constant $C(N) \geq 5$ in Lemma \ref{l.odeT}, we are fine. Let us call this event $\tilde A_U:= \{ H(U) \leq 5 \}$.

\ni
Back to the case where $k$ is any scale $\frac L 2  \leq U \leq k \leq L$. Recall that Lemma \ref{l.ext} together with Lemma \ref{lemma:projection_on_sphere_trhough_a_good_point} imply that, given $\sigma_{\md \p \Lambda_k}$, 
\begin{equation} \label{eq:upperbound2}
\DE_{\Lambda_k}(\sigma^{opt}) =: GS_{\Lambda_k}(\sigma_{\md \p \Lambda_k})   \le   C   \sqrt{k} \sqrt{\text{DE}_{\partial \Lambda_{k} }(\sigma)}   ~. \end{equation}

As in the deterministic case, we thus have
\[ H(k) - H(k - 1) = \sum_{e \in \partial \Lambda_k } (\nabla_e \sigma)^2 + \sum_{\substack{e = \{x,y\}, \\  x \in \partial \Lambda_{k-1}, y \in \Lambda_{k} }} (\nabla_e \sigma)^2  \ge \text{DE }_{\partial \Lambda_k}(\sigma) \]



As such, if the event $A_{k}$ holds (and notice it does with high conditional probability  given $\sigma_{| \p \Lambda_{k}}$ when $\beta$ is high enough), then 
\[
H(k ) - H(k-1) \ge \frac{\Big(GS_{\Lambda_k}(\sigma_{\md \p \Lambda_k}) \Big)^2}{C^2 k } 
\overset{\text{event }A_{k}}\geq  \frac 1 {25 C^2 k} H(k)^2
\]
\margin{C: This step would in fact not work well at all with a Markovian Definition of $H(k)$ so our approach seems rigid here. }

As in Subsection \ref{ss.ODE}, dividing by $H(k) H(k-1)$ and summing from $k=L$ to $k=\frac L 2$ we get on the event 
\begin{align*}  
\{ N=0\}  \bigcap  \bigcap_{k= \frac L 2}^L A_k
\end{align*}
 the following lower bound on $H(\frac L 2)^{-1}$: 

\[ \frac{1}{H(L/2)} \ge \frac{1}{H(L/2)} - \frac{1}{H(L)}  \ge \hat C \sum_{k = L/2 }^L \frac{1}{k}  \geq \hat C(N)\,.
\]
We also noticed that on the event $\{ U \neq 0\} \bigcap \tilde A_U$, we have $H(\frac L 2) \leq H(U) \leq 5$. 
Since the probability of 
\begin{align*}  
\Big(
\{U = 0\} \bigcap \big( \bigcup_{k= \frac L 2}^L  A_k^c \big) \Big)  \bigcup  \Big(\{U\neq 0\} \bigcap  \tilde A_U^c \Big)
\end{align*}
can be made arbitrary small (when choosing $\beta$ sufficiently large), by combining Markov property with union bound, this concludes the proof Lemma \ref{l.odeT} with $C(N):= 5 \vee \hat C(N)^{-1}$. \qed

\medskip

\ni
{\em End of the proof of Theorem \ref{th.SG}.}
At this stage, notice we did not yet fix the scale $L=L(\delta)$ (which ultimately impacts how large $\beta=\beta(\delta)$ needs to be).

Let us briefly explain how to conclude the proof of Theorem \ref{th.SG}.
A straightforward consequence of Lemma \ref{l.odeT} is that  uniformly over boundary conditions, one has 
\begin{align*}  
\FK{\Lambda_L,\beta}{\bar \sigma}{GS_{\Lambda_{L/2}}(\sigma_{\p \Lambda_{L/2}}) \leq C(N)} \geq 1 -\alpha\,.
\end{align*}

The advantage of this expression is that it is more Markovian. (It does not involve the yet to be explored spin system inside $\Lambda_{L/2}$).

We claim that by adapting the deterministic dyadic argument from Subsection \ref{ss.HR} to this random setting as we did above, i.e.  by using repeatedly the same Markovian exploration together with Lemma \ref{l.T}, we conclude the proof of Theorem \ref{th.SG}. 
One significant different is that in this second step of the proof, we may no longer set $\gamma^2:=1$ in Lemma \ref{l.T} but instead the much smaller value $\gamma^2:=\delta^2$. (N.B. This could  be tuned in a more efficient way along dyadic scales if one were looking for a better quantitative control of $\beta=\beta(\delta)$). Finally, $L(\delta)$ needs to be fixed in this second step so that, as in Subsection \ref{ss.HR} in the deterministic case, we have enough scales to reduce the gradient at one edge. Slightly more scales are needed in the random case, as one looses an additional  multiplicative error at each step as in the above Markovian analysis.  Once $L(\delta)$ is chosen at this very final step, it influences the choice of $\beta=\beta(\delta)$ so that all previous probabilistic steps and union bounds over all edges in $\Lambda_{L/2}$ are simultaneously  valid with probability at least $1-\delta$. 

\qed

\subsection{Quantitative bound on $L(\delta)$ and $\beta(\delta)$.}\label{ss.quant}
The following quantitative version of Theorem \ref{th.SG} holds.
\begin{proposition}\label{th.SGquant}
Let $d=2$ and $N \geq 3$. 
There exists some number $c > 0$ that depends only on $N$ such that, for all $0 < \delta< \frac 12 $, all $L \ge \delta^{-c}$  and all $\beta \ge 1$, there exists some $C  = C(N) > 0$ such that uniformly over all boundary condition $\bar \sigma \in (\S^{N-1})^{\p \Lambda_L}$, one has 
\begin{align*}  
\FK{\Lambda_L, \beta}{\bar \sigma}{ \exists i\sim j  \in \Lambda_{L/4} \, \text{ s.t. } \|\sigma_i - \sigma_j\|_{2,\R^N} > \delta} \le \exp(-  \frac 1 C  \beta \delta^2  - C \ln(\delta) L^2) \,.
\end{align*}

\end{proposition}

In particular, if $\frac{\beta}{L^2} \ge C^2 \frac{\ln(\frac 1 \delta)}{\delta^2}$ and $L \ge \delta^{-c}$, then the conclusion of Theorem \ref{th.SG} holds.

To see this, we follow the steps of the proof of Theorem \ref{th.SG} and give a quantitative bound for each of them.

Notice first that Lemma \ref{l.T} holds with
\begin{align*}  
\FK{\Lambda, \beta}{\bar \sigma}{ \DE_\Lambda(\sigma)>  4\, GS_{\Lambda}(\bar \sigma)  + \gamma^2 } \le \exp(- \frac 12 \beta \gamma^2 - C \ln(\delta ) |\Lambda| )\,.
\end{align*}
 as can be seen by looking at the last step of its proof.

 Lemma \ref{l.odeT} holds with 

 \begin{align*}  
\FK{\Lambda_L,\beta}{\bar \sigma}{\DE_{\Lambda_{L/2}}(\sigma) \leq C(N)} \geq  1 - \exp(- \frac 12 \beta - C_1 \ln(\delta ) L^2)\,.
\end{align*}

Indeed, in the proof of Lemma \ref{l.odeT} we apply Lemma \ref{l.T} at most $L$ times with $\gamma =1$ on scale that vary between $L$ and $\frac L 2$ and do a union bound on each of these steps. 

As for the last step  of the proof, we need the energy to go down from $C(N)$ to $\delta$ which requires $ C\log(\frac 1 \delta)$ applications of the probabilistic version of Lemma \ref{l.dyadic}. In particular, $L$ needs to be greater or equal than $\delta^{-c}$ for some $c > 0$ that depends only on $C(N)$ and the constant $\lambda$ of Lemma \ref{l.dyadic}. At every step we use Lemma \ref{l.T} with $\gamma = c \delta$ to ensure that the multiplicative shrinking occurs with probability at least $1 -\exp( - C \beta \delta^2 + C_1 \ln(\frac 1 \delta ) L^2)$. We need to apply this bound $c \ln(\frac{1}{\delta})$ times and apply a union bound. This gives an upper bound in
\[  c \ln(\frac 1 \delta ) \exp( - C \beta \delta^2 + C_1 \ln(\frac 1 \delta ) L^2) = \exp( - C \beta \delta^2 + C_2 \ln(\frac 1 \delta ) L^2)\] for the probability that the energy is larger than $\delta$ in $\Lambda_{\frac L 4 }$ given that it is smaller than $C(N)$ in $\Lambda_{\frac L 2 }$. Applying union bound again, we get 

\begin{align*}  
\FK{\Lambda_L, \beta}{\bar \sigma}{ \exists i\sim j  \in \Lambda_{L/4} \, \text{ s.t. } \|\sigma_i - \sigma_j\|_{2,\R^N} > \delta} \le \exp(-  C \beta \delta^2  - C \ln(\delta) L^2) \,.
\end{align*}

\subsection{When $d \geq 3$.}
The proof of Proposition \ref{pr.NoSurface} is identical to the first part of the proof of Theorem \ref{th.SG}. Indeed, it consists in exploring from layer $\p \Lambda_n$ down to layer $\p \Lambda_{\frac n 2}$ by applying Markov property at each successive $\p \Lambda_{k}$, $ \tfrac n 2 \leq k \leq n$. 

The first step of the Proof of Theorem \ref{th.SG} corresponds to the probabilistic Lemma \ref{l.odeT}. Its proof uses two main ingredients, (1) the probabilistic Lemma \ref{l.T} whose proof is insensitive to the dimension $d$ and (2) the deviation from the deterministic ODE from Section \ref{ss.ODE}. The way to handle the deviation from the ODE is the same in dimension $d\geq 3$, but the ODE itself of course is different. As such we obtain similarly as in $d=2$ that for a suitable constant $C$ (slightly larger than what the deterministic result would give), with probability going to 1 as $\beta\to \infty$, the event 
\[
\left\{ \text{DE}_{\Lambda_{n/2}}(\sigma) \le \begin{cases} C n^{d-2} & \text{if } N \ge 3 \\  C n^{d-2} \log^2(n) & \text{if } N = 2 \,.\end{cases}  
\right\}
\]
is satisfied.

\section{Non-trivial harmonic maps $\sigma : \Z^2 \to \S^1$}\label{s.VortexS1}

In this section, we prove Proposition \ref{pr.ExS1}. Its proof was suggested to us by ChatGPT.

\begin{remark}  
The construction easily extends to a non-trivial global minimizing map from $\sigma : \Z^d \to \S^1$ for any $d\geq 2$. 
\end{remark}

\ni
{\em Proof of Proposition \ref{pr.ExS1}.}


The approach we use is based on the work of Garroni Petrache and Spadaro \cite{garroni2024clearingoutdipolesminimisers2dimensional} in which they analyse vortices of $\S^1$ valued functions on finite planar graph. Consider a finite planar graph $G = (V,E)$ that has a boundary set $\partial V$ which is the set of vertices around a face $f_{\text{ext}}$ that we see as the exterior face. Let $\psi: V \to \R$ be a function and $\sigma: V \to \S^1 $. We say that $\psi$ is a lift of $\sigma$ if the projection of $\psi$ on $\S^1 = \R/(2 \pi\Z)$ equals $\sigma$. If $x \in \R$, we denote by $d(x,\Z)$ its signed distance to an integer. That is $d(x,\Z)= x - k_x$ where $k_x$ is the closest integer to $x$. When $x$ is a half integer, we set $d(x,\Z) = \frac 1 2$ if $x > 0 $ and $d(x,z) = - \frac 1 2$ if $x < 0$. This choice ensures that $d(x,\Z) = - d(-x, \Z)$ for any $x$. We say that $\psi$ has a vortex at the face $f$ if $\curl \psi(f) := \sum_{e = (e_-, e_+) \in \p f} d(\psi(e_+) - \psi(e_-), \Z) \neq 0$ where the sum is performed over the edges of $f$ in counterclockwise orientation. By a slight abuse of notation we will denote $\curl(\sigma) = \curl( \psi)$ if $\psi$ is a lift of $\sigma$. This can create ambiguity only if some of the gradients along edges $e = (e_+ , e_-)$, $|\gamma(e_+) - \gamma(e_-)| = \frac{1}{2}$. The main result of \cite{garroni2024clearingoutdipolesminimisers2dimensional} on the XY model can be informally stated as follows. Consider a boundary condition $\sigma: \partial V\to \S^1$. Assume that $\sigma$ revolves around the outer face exactly once,\margin{maybe define this ?} always going in the counterclockwise direction. Then there exists a harmonic map $\hat \sigma : V\to \S^1$ and a face $f_0$ such that $\sigma$ has a unique vortex at the face $f$ with $\curl(\sigma) f = 1$ inside $V \backslash \p V$ which is located at $f_0$. 
This result does not seem to suffice to deduce the existence of a non-constant harmonic map over the whole $\Z^2$ from this result. We will rather make use of Theorem 4.2 of \cite{garroni2024clearingoutdipolesminimisers2dimensional}  which is an intermediary step in their paper. Let $G = (V,E)$ be a bidirectional finite graph (not necessarily planar), that is to say an oriented graph that is such that all edges appear in both directions. If $e$ is an edge we denote by $e_+ \in V$ its beginning and by $e_-$ its end. We denote by $-e$ the edge in the reversed direction that is if $e = (e_+, e_-)$, then $-e =  (e_- , e_+)$.
Recall that if $\gamma: E \to \R$ satisfies $\gamma(-e) = - \gamma(e)$ for all $e \in E$, we say that $\gamma$ is a one-form on $G$. We denote by $\div(x) = \sum_{e, e_+ = x} \gamma(x)$. We say that $\gamma$ has integral divergence if for all $x \in V \backslash \p V$, $\div(x) \in \Z$. We denote by $\| \gamma \|_{\p} = \sum_{x \in \p V } \sum_{v \sim x, v \notin \p V} |\gamma((x,v))|$. We denote by $\Flux(\gamma) = \sum_{x \in \p V } \sum_{v \sim x, v \notin \p V}  \gamma(x,v)$

\begin{theorem}[Theorem 4.2 of \cite{garroni2024clearingoutdipolesminimisers2dimensional}]
  \label{theorem_integer_divergence}
Let $\alpha$ be a one-form with integral divergence with $\Flux(\alpha) \in {\pm 1}$ and $\|\alpha\|_{\partial} = 1$. Then, there exists a one-form $\gamma$  and a vertex $v_0$ of $V \backslash p V such that $ such that, for all $e \in E$, $|\gamma(e)| \le |\alpha(e)|$ and $\gamma(e) = \alpha(e)$ for all edges from $\partial V \times V$ and such that $div(\gamma)(x) = - \Flux(\alpha) \delta_{x_0}(x)$.
\end{theorem}

Let us explain how to deduce the existence of a global harmonic map.\footnote{This proof was found using ChatGPT Astra. However, the formulation below has been written by the authors.} Let $G_n = \Z^2/n \Z^2 \cup \{v_0\} $ denote the $n \times n $ square torus to which we append one special vertex $v_0$ which we connect to all vertices in $\Z^2/n \Z^2$. We set $\p V = \{v_0\}$. Consider the energy functional \[E(\gamma) = \frac{1}{2}\sum_{e \in E, e_+ \neq v_0 , e_- \neq v_0} ( 1 - \cos( 2 \pi \gamma(e))\] defined on the set of one-forms on $G_n$. 
Let $\mathcal{M}_n$ denote the set of one-forms $\gamma$ with integral divergence and that are such that $\gamma((v_0,x)) = -\frac{1}{n^2}$ and that satisfies $|\gamma(e)| \le \frac 1 2 $ for every edge $e$. This set is non-empty. Indeed, let $\Delta$ be the Laplace operator on real valued functions defined on the vertices of $\Z^2/n \Z^2$. Let $\Delta^{-1}$ be its inverse on the set of functions $f: \Z^2/n \Z^2 \to \R$ that satisfy $\sum_{x \in \Z^2/n \Z^2 } f(x) = 0$. Let $x_0 = (0,0)$. For $e$ inside the torus, let $\gamma(e) =  \Delta^{-1} (\delta_{x_0} - \frac{1}{N^2}) (e^+) - \Delta^{-1} (\delta_{x_0} - \frac{1}{N^2}) (e^-)$. We extend this one-form by $\gamma((v_0,w))= - \frac{1}{n^2}$ for any $w$ in $\Z^2/n \Z^2$. The one-form $\gamma$ has integral divergence. To see that the gradient of $\gamma$ is bounded by $\frac 12$, one can use the maximum principle on the function $(\delta_{x_0} - \frac{1}{N^2}) (\cdot ) - (\delta_{x_0} - \frac{1}{N^2}) (\cdot + e ) $ where $e \in \{\pm (1,0), \pm (0,1)\}$ which is harmonic on the torus except at $x_0$ and $x_0 - e $ and takes the value $\frac 1 4 - \frac 1 N^2$ on $x_0$ and $x_0 - e$.

As $\mathcal{M}_n$ is compact, the functional $E$ has a minimizer in $\mathcal{M}_n$. Any one-form $\alpha \in \mathcal{M}$ satisfies $\Flux(\alpha) = \|\alpha\|_{\p} =1$. We may thus apply Theorem \ref{theorem_integer_divergence} to any $\alpha \in \mathcal{M}$ to  obtain a one-form $\gamma \in \mathcal{M}$ that has exactly one vertex with non-zero integral divergence and that satisfies $E(\gamma) \le E(\alpha)$ as $x \mapsto 1 - \cos(x)$ is a nondecreasing function in $[0, \frac{1}{2}]$.
 Thus, there exists a minimizer $\gamma_0 = \gamma_0^n$ of $E$ over $\mathcal{\gamma}$ that has exactly one vertex in $\Z^2/n \Z^2$ on which $\gamma$ has non-zero integral divergence. By invariance by translation in the torus, we may assume that this vertex is $(0,0)$. Let $\gamma_0^{\infty} $ be a one-form defined on $\Z^2$ that is obtained by taking a subsequential limit of $(\gamma_0^n)_{n \ge 0}$ for the topology of pointwise convergence on edges of $\Z^2$. Note that $\gamma_0^{\infty}$ has $0$ divergence everywhere except at the face $(0,0)$ where it must have divergence $1$.
  We can now use $\gamma_0^{\infty}$ to define a non-trivial harmonic map $\sigma$ on the dual graph of $\Z^2 $ (that is $\Z^2 + (\frac 12,\frac 12)$). We set $\sigma(\{\frac 12, \frac 12\}) = 0$, and we then define $\sigma(x) = \prod_{e \in c_{x}}\exp( 2 i\pi  \gamma_0^{\infty}(e^{\dagger}))$ where $c_x$ denotes any path in the dual of $\Z^2$ between $[-\frac 12, \frac 12]$ and $x$ and $e^{\dagger}$ denotes the dual edge of $e$.
  This is well-defined as $\gamma_0^{\infty}$ has integral divergence. The map $\sigma$ is a global harmonic map. Indeed, assume that $\sigma'$ is a finite modification of $\sigma$ that coincides with $\sigma$ outside the dual of $\Lambda_{n_0}$. Let us write $\sigma'(x) = \sigma(x) \exp( 2 i \pi  h_x)$ for some $h_x \in (-\frac 12, \frac 12 ]$ and set $\gamma'(e)$ as the  representative in $(-\frac 12 , \frac 12]$ of $ \gamma(e) + h_{e_+} - h_{e_-}$ mod $\Z$. For $n \ge 1$, we define similarly form $\gamma'_n(e)$ on $G_n$ as the representative mod $\Z$ of $\gamma_n(e)+ h_{e^{\dagger}_+} - h_{e^{\dagger}_-}$ for edges that do not have $v_0$ as an endpoint, and $\gamma'_n((v_0,x)) = - \frac{1}{n^2}$. 
   Note that $\gamma'_n$ is in $\mathcal{M}_n$. Indeed, for any vertex $v \in  \Z^2/n \Z^2$, we have modulo $\Z$, $\div \gamma'_n(v) = \div \gamma_n(v) + \sum_{e , e_+  = v  } (h(e_+) - h(e_-))= \div \gamma_n(v) $. We thus have $E_n(\gamma_n) \le E_n(\gamma'_n) $ as $\gamma_n$ is a minimizer. 
   Now, notice that  $  E_n(\gamma_n) - E_n(\gamma'_n) = 2 \sum_{e , e_{+ } \in \operatorname{Supp}(h)} \cos(2 \pi (\gamma_n(e) )) - \cos(2 \pi (\gamma_n(e) + h_{e_+} - h_{e_-}))$ for $n \ge n_0$. This sum has a fixed number of term that does not depend on $n$. Passing to the limit, we thus get \[ \operatorname{DE}_{\Lambda_{n_0}}(\sigma) - \operatorname{DE}_{\Lambda_{n_0}}(\sigma' ) =  \sum_{e , e_{+ } \in \operatorname{Supp}(h)} \cos(2 \pi (\gamma_n(e) )) - \cos(2 \pi (\gamma_n(e) + h_{e_+} - h_{e_-}))  \le 0 ~.\]

 Let us now show that the non-trivial harmonic map we have build has Dirichlet energy at least $C \log(n)$ in $\Lambda_n$. Let $\gamma = \gamma_0^{\infty}$. It suffices to show that $\sum_{e^{\dagger} \in \overrightarrow{\partial \Lambda_k}} ( 1 - \cos(2 \pi \gamma(e) ))  \ge \frac{c}{k}$ for any $k \ge 1$ where the sum runs over edges of $\partial \Lambda_k$ in the clockwise direction. As, for any $x \in [- \frac{1}{2}, \frac{1}{2}]$, we have $1 - \cos(2 \pi x) \ge 2 \pi^2 x^2$, it suffices to show that $\sum_{e^{\dagger} \in \overrightarrow{\partial \Lambda_k}} \gamma(e)^2 \ge \frac c k $

 This follows from Cauchy--Schwarz inequality and the fact that $\gamma$ has zero divergence at every vertex except at the origin $x_0$ where it has divergence $1$. Indeed,
 \[ \sum_{e^{\dagger} \in \overrightarrow{\partial \Lambda_k}} \gamma(e)^2 \ge \frac{1}{8 k +4}  \sum_{e^{\dagger} \in \overrightarrow{\partial \Lambda_k}} |\gamma(e)| \] as there are $8 k +4$ edges in $\partial \Lambda_k$.

 By the discrete divergence theorem, \[\sum_{e^{\dagger} \in \overrightarrow{\partial \Lambda_k}} \gamma(e) = \sum_{x \in \Lambda_k} \div(\gamma)(x) = \div(\gamma)(x_0) = 1\, . \] Now as, $\sum_{e^{\dagger} \in \overrightarrow{\partial \Lambda_k}} |\gamma(e)| \ge 1$, we obtain
$\sum_{e^{\dagger} \in \overrightarrow{\partial \Lambda_k}} \gamma(e)^2 \ge  \frac{1}{8 k +4}.$

\section{Open problems}

\begin{OP}\label{op.logn}
In Theorem \ref{th.DE} when $N=2$, for any $d\geq 2$ we have a lower bound in $n^{d-2} \log(n)$ thanks to Proposition \ref{pr.ExS1}
But the upper bound in Theorem \ref{th.DE} has an additional $\log(n)$, it is in $n^{d-2} \log^2(n)$. What is the optimal scaling of the energy in the bulk for $N=2$ lattice harmonic maps ?
\end{OP}

\begin{OP}\label{op.vortex}
Prove that Theorem \ref{th.SG} does not hold for the $2d$ XY model ($d=2$ and $N=2$). See Figure \ref{f.XY}.
\end{OP}

\begin{OP}\label{op.perco}
Even though Theorem \ref{th.SG} should not apply for $XY$ Gibbs measures (open problem above), what about Corollary \ref{c.SG2} ? We believe that maybe surprisingly it should still hold. Of course the mechanism would be different (one would not be able to extract it from Theorem \ref{th.SG}). It would have a chance to be correct if $\beta$ is so large that the density of vortices is low enough to have zero vortices in a $\Lambda_L$ box with high probability. In this case, Corollary \ref{c.SG2} would have a chance to hold.
\end{OP}

\bibliographystyle{alpha}
\bibliography{references-S2}

\end{document}